\RequirePackage{fix-cm}
\documentclass[12pt]{article}
\title{Boundary of the action of Thompson's group F on dyadic numbers}
\usepackage{latexsym}
\usepackage{graphicx}
\usepackage{color}
\usepackage{imakeidx}
\makeindex[name=sub,title=Subject Index]
\makeindex[name=foo,title=Notation Index]
\usepackage{amssymb,amsmath,amscd,epsfig,amsthm,multirow}
\usepackage{amsfonts,latexsym,xy}
\xyoption{all}

\usepackage{enumerate}
\usepackage{makeidx}
\usepackage{mathtools}
\usepackage{bbm, epstopdf}
\usepackage{comment}

\newtheorem{theorem}{Theorem}[section]

\newtheorem{proposition}[theorem]{Proposition}

\newtheorem{remark}[theorem]{Remark}

\newcommand{\supp}{\mathop{\mathrm{supp}}}

\newcommand {\NN}{\mathbb{N}}
\newcommand {\PP}{\mathbb{P}}
\newcommand {\DD}{\mathbb{D}}
\newcommand {\ZZ}{\mathbb{Z}}

\makeindex
\begin{document}
\title{Boundary of the action of Thompson's group F on dyadic numbers}
\author{Pavlo Mishchenko\footnote {Ecole Normale Sup\'erieure de Lyon, supported by Labex Milyon}}
\date{\today}
%\institute{P. Mishchenko \at ENS Lyon
              %\email {mishchenko.p@gmail.com}
              %}
\maketitle
\begin{abstract}
We prove that the Poisson boundary of a simple random walk on the Schreier graph of action $F \curvearrowright \DD$, where $\DD$ is the set of dyadic numbers in $[0, 1]$, is non-trivial. This gives a new proof of the result of Kaimanovich: Thompson's group $F$ doesn't have Liouville property. In addition, we compute growth function of the Schreier graph of $F \curvearrowright \DD$.
\end{abstract}

\section{Introduction}
Let $G$ be a group equipped with a probability measure $\mu$.
A right random walk on $(G, \mu)$ is defined as a Markov chain $Z$ with the state space $G$ and
transitional probabilities $\PP(Z_{n+1} = g| Z_n = h) = \mu(h^{-1}g)$.
Specifying initial measure $\theta$ (distribution of $Z_0$), we obtain a probability measure $\PP_{\theta}^{\mu}$ on the space of trajectories $(Z_i)_{i \geq 0} \in G^{\ZZ_+}$. Usually one takes $\theta = \delta_e$ - Dirac measure at the group identity.
The Poisson boundary of the pair $(G, \mu)$ can be defined as the space of ergodic components of the time shift on the $(G^{\NN}, \PP_{\delta_e}^{\mu})$ \cite{mazur}. For more equivalent definitions of the boundary one can look at  \cite{kaimversh}. A pair $(G, \mu)$ is said to have \textit{Liouville property} if the corresponding Poisson boundary is trivial, or, equivalently, when the space of bounded $\mu$-harmonic functions on $G$ is 1-dimensional, i.e. consists of constant functions. A group $G$ has \textit{Liouville property} iff for every symmetric, finitely supported $\mu$ the pair $(G, \mu)$ does. For a recent survey and results on Liouville property and Poisson boundaries see \cite{Anna}, \cite{IET} and \cite{mattebon}.

In this note we prove that Richard Thompson's group $F$ doesn't have Liouville property. A survey on Thompson's groups is presented in \cite{cfp}. Here we only mention that question of amenability of $F$ is one of the major open problems now.

\section{Main results}
Consider a simple random walk on a locally finite graph $G = (V, E)$. Fix a starting point $x_0$. This enables trajectory space $V^{\ZZ_+}$ with a probability measure $P$.
The notion of the boundary is easily adapted to this case: it is the space of ergodic components of the time shift on the $(V^{\ZZ_+}, P)$.
We'll use electrical networks formalism as it appears in \cite{woess}.
Throughout the paper, $d(\cdot, \cdot)$ will denote
standard graph distance.
Let $B(x, n) = \{y \in V: d(x,y) \leq n  \}$ - ball centered at $x$ of radius n. Define also $ \partial B(x, n) = \{y \in V: d(x,y) = n  \}$.
\begin{theorem}\label{exp_geod}
 Suppose a locally-finite graph $G$ is given. Fix any vertex $x_0$. Let $\Upsilon(x_0)$ be the set of geodesics starting at $x_0$. Define $\texttt{gd}(x, n) = \#\{ \gamma=[x_0,...,x_m] \in \Upsilon(x_0):\, x \in \gamma\, \text{  and } d(x, x_m) =n\}$. Suppose there exist some real numbers $c, C > 0$ and $q>1$ such that the following conditions are satisfied:

\begin{subequations}\label{cond}
\begin{align}
\texttt{gd}(n, x) &\leq Cq^n \text{  for every }x \in X\\
cq^n &\leq \texttt{gd}(x_0,n)
\end{align}
\end{subequations}

Then $cap(x_0) > 0$.
\end{theorem}

\begin{remark}
For example, it's easy to see that conditions(\ref{cond}) are obviously satisfied for a regular $m$-tree, with $q = m$.
\end{remark}

\begin{proof}
We have to show that if $f \in D_0(N)$, $f(x_0) = 1$ then it's Dirichlet norm is bounded from below by some positive constant.
Let $n$ be such that $\supp f \subseteq B(x_0, n-1)$. All resistances are equal to 1 in our case, so we may write
\begin{align*}
D(f) =
\sum\limits_{e \in E}{(f(e^+) - f(e^-))^2} &\geq
\sum_{k=0}^{n-1}\sum_{\substack{x \in \partial B(x_0,k) \\ y \in \partial B(x_0, k+1) \\(x,y) \in E}}{(f(x)-f(y))^2} &=\\
\sum_{\substack{\gamma \in \Upsilon(x_0) \\ [x_0,..,x_n] = \gamma}}\sum_{k=0}^{n-1}{\frac{(f(x_k) - f(x_{k+1}))^2}{\texttt{gd}(x_{k+1}, n-k-1)}} &\geq
\sum_{\substack{\gamma \in \Upsilon(x_0) \\ [x_0,..,x_n] = \gamma}}\frac{(\sum_{k=0}^{n-1}{f(x_k) - f(x_{k+1})})^2}{\sum_{k=0}^{n-1}{\texttt{gd}(x_{k+1}, n-k-1)}} &\geq\\ \texttt{gd}(x_0,n)\frac1{\sum_{k=0}^{n-1}{Cq^{n-k-1}}} &\geq
cq^n\frac1{\sum_{k=0}^{n-1}{Cq^{n-k-1}}} \geq \frac{c(q-1)}{C}
\end{align*}
For the first inequality, we cancel edges which connect vertices which connect vertices at the same distance from $x_0$.
For the second equality, we consider geodesics from $x_0$ to points at the distance $n$ and sum quantities $\frac{(f(x_k) - f(x_{k+1}))^2}{\texttt{gd}(x_{k+1}, n-k-1)}$ over them, getting $(f(x_k) - f(x_{k+1}))^2$ by definition of $\texttt{gd}$.
We use next Cauchy-Schwartz and finiteness of support of $f$: $f\equiv 0 $ outside of $B(x_0, n-1)$.
Thus we have $$D(f) \geq \frac{c(q-1)}{C} $$ for every finitely-supported f,
hence $cap(x_0) > 0$
\end{proof}

This implies, by theorem (2.12) from \cite{woess}, that simple random walk on $G$ is transient.
Now we are going to establish a theorem which connects transience of certain random walks to non-triviality
of boundary. Following \cite{kaimversh}, we call subset $A \subset G$ a trap, if $\lim_n\mathbbm{1}(Z_n \in A)$ exists for almost all trajectories $Z \in G^{\NN}$. We call a graph transient if the simple random walk on it is transient.

\begin{theorem}\label{transtriv}
Let $T$ be a tree with a root vertex $v$ such that for each descendant $v_1, \dots, v_n$ of $v$ ($n \geq 2$)
a subtree $T_i$ rooted at $v_i$ is transient. Then the boundary of simple random walk on $T$ is nontrivial.
\end{theorem}
\begin{proof}
Take any $T_i$.
Almost surely, every trajectory hits $v$ only finitely many times. The only way to move from $T_i$ to $T_j$ is to pass by $v$. Therefore, for any $i$, we'll stay inside or outside of $T_i$ from some moment. This means that $T_i$ is a trap. Let' s prove that it is nontrivial, i.e. random walk will stay at $T_i$ with positive probability. If this is true for each $i$, then every $T_i$, $1 \leq i \leq n$, is a nontrivial trap, so boundary is indeed nontrivial. Consider the following set of trajectories of the simple random walk on $T$:
$$A = \{Z: Z_1 = v_i, \forall k \geq 2\,\, Z_k \neq v_i\}.$$
In addition, consider the set of trajectories of the simple random walk on $T_i$: $$A' = \{Z': Z'_0 = v_i, \forall i \geq 1\,\, Z'_i \neq v_i\}.$$Collecting the following facts:\\
-simple random walk on $T$ goes to $v_i$ with probability $1/n$;\\
-probability of going from $v_i$ not to $v$ is $\frac{\deg(v_i) -1}{\deg(v_i)}$;\\
-$(Z_{k+1})_{k\geq0} \in A'$, and transition probabilities are the same for $Z_{i+1}$ and $Z'_i$ for $i \geq 1$\\
we obtain
$$\PP(A) = \frac1n\frac{\deg(v_i) -1}{\deg(v_i)}\PP(A').$$
This shows us that indeed $\PP(A) > 0$, as $\PP(A') > 0$ due to transience.
\end{proof}

\begin{proposition}\label{nonliouv}
Let $H$ be a graph which is formed by adding a set of graphs $G_v$ with pairwise disjoint sets of vertices to each vertex $v$ in $T$. Then the boundary of simple random walk on $H$ is nontrivial.
\end{proposition}
\begin{proof}
Bounary of $T$ is nontrivial, so we have non-constant bounded harmonic function $h$ on $T$. We can extend it to the whole $H$ by setting
$$\hat{h}(x) =\begin{cases}
h(x),& \text{if } x \in T\\
h(v),& \text{if } x \in G_v
\end{cases}
$$
This way we get non-constant bounded harmonic function $\hat{h}$ on $H$, so boundary is non-trivial.
\end{proof}

Recall that Richard Thompson's group $F$ is defined as the group of all continuous piecewise linear transformations of $[0,1]$, whose points of non-differentiability belong to the set of dyadic numbers and derivative, where it exists, is an integer power of 2. It is known to be 2-generated.
Now we are ready to prove the main theorem.
\begin{figure}[ht]
\begin{center}
\hspace{1cm}\epsfig{file=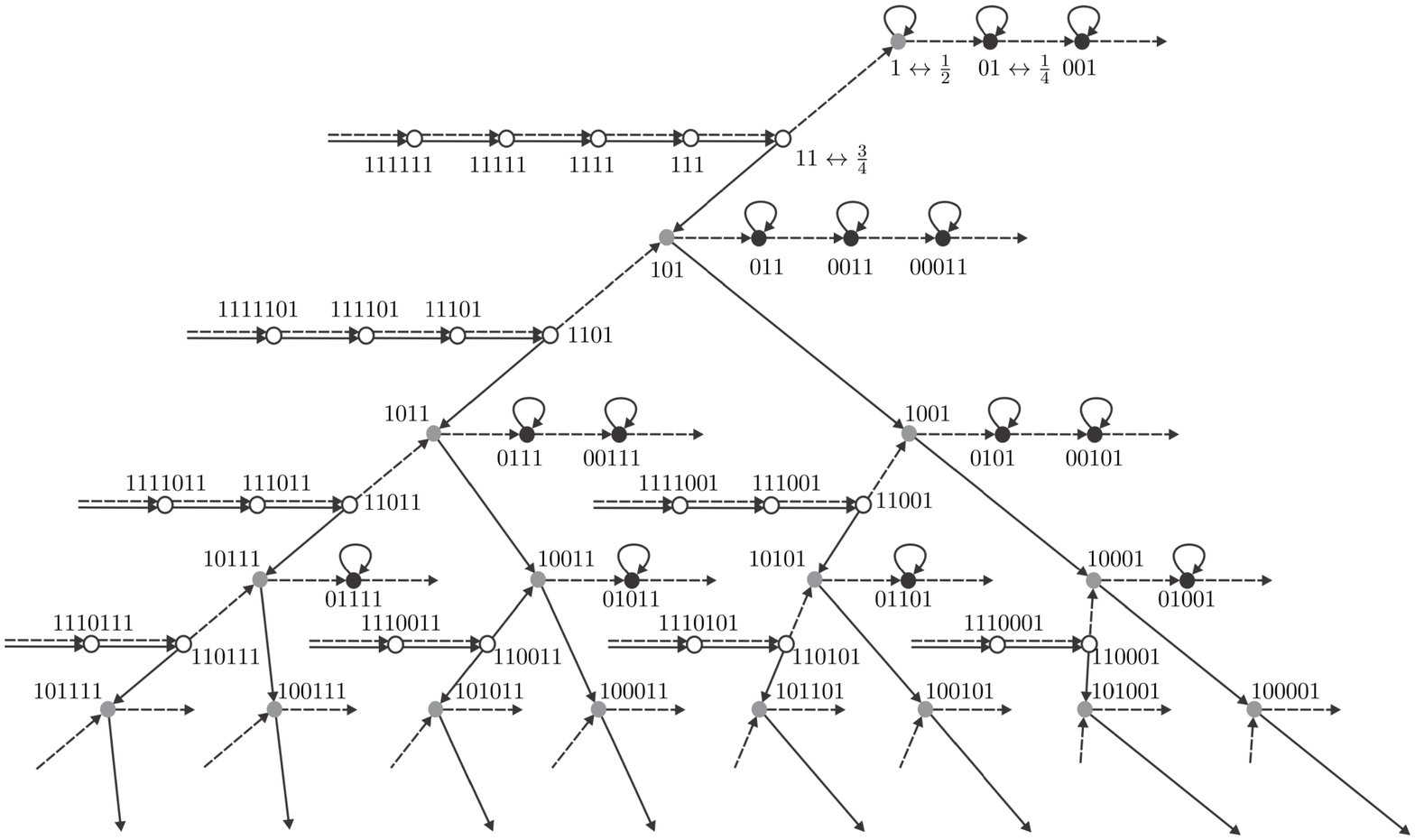,width=440pt}
\end{center}
\caption{Schreier graph $\mathcal{H}$ of the action of $F$ on the orbit of $1/2$\label{fig_schreier_dyadic}}
\end{figure}
\begin{theorem}
Thompson's group $F$ does not have Liouville property.
\end{theorem}
\begin{proof}
First of all, we observe that if action $G \curvearrowright X$ is non-Liouville, i.e. there are bounded non-constant harmonic functions on the Schreier graph of this action, then $G$ itself is non-Liouville. Required harmonic function on $G$ is just a pullback of a harmonic function $h$ on $X$: $h'(g) = h(g.x)$. Obviously, $h'$ is harmonic if $h$ is.
We consider the action of $F$ on the set of all dyadic numbers in $[0,1]$. We use presentation of its Schreier graph $\mathcal{H}$ constructed by D. Savchuk in \cite{savchuk}. It is illustrated on the Figure 1. We look at the tree $T$ rooted at 101 formed by grey vertices and white which are connected with the grey ones. We need to verify that $T$ satisfies condition \ref{cond}. Take $x_0$ to be the point $3/8 = 101$. For grey vertices $x$ we have $\texttt{gd}(n, x) = | \partial B(x_0, n)|$, and for white $y$ we have $\texttt{gd}(n, y) = |\partial B(x_0, n-1)|$. In fact, $|\partial B(x_0, n)|$ may be calculated explicitly ($T$ is a famous Fibonacci tree), and value $q = \frac{1+\sqrt{5}}{2}$ works. Hence, we can apply consequently \ref{exp_geod}, \ref{transtriv} and \ref{nonliouv} to see that there are non-constant bounded harmonic functions on $\mathcal{H}$, so, by the remark in the beginning of the proof, on the Thompson's group $F$.
\end{proof}
\begin{remark}
The fact that simple random walk on the Thompson's group $F$ has nontrivial boudary is first proven by Kaimanovich in \cite{kaimunpub}.
\end{remark}

\section{Growth function of $\mathcal{H}$}
In \cite{grignagn} different types of growth functions for groups are defined. We adapt these definitions to Schreier graphs of group actions. We'll compute growth function of $\mathcal{H}$.
Suppose we have a Schreier graph of action of a group $G$ on set $X$. Fix some starting point $p \in X$.
A cone type of a vertex $x$ is defined as follows:
\begin{flalign*}
    C(x) = \{ g \in G: \phantom{\hspace{12cm}}\\
\nonumber \text{if $w$ is a geodesic from $p$ to $x$, then $wg$ is a geodesic from $p$ to $g(p)$}\}. \phantom{\hspace{1cm}}\nonumber
\end{flalign*}
\textit{Complete geodesic growth function} is defined as
$$L(z) = \sum\limits_{g \in G: g \text{  is a geodesic for } g(p)}{gz^{|g|}}.$$
\textit{Geodesic growth function} is defined by sending all group elements to $1$, namely
$$l(z) = \sum\limits_{g \in G: g \text{  is a geodesic for } g(p)}{z^{|g|}}.$$
\textit{Orbit growth function} is defined as
$$\widehat{l}(z) = \sum\limits_{n=0}^{\infty}\#\{x \in X: \exists g \in G - geodesic: |g| = n, g(p)=x\}{z^n}.$$
Now consider $\mathcal{H}$. We are interested in geodesics starting at point $p = 1/2 (100...)$. Then we have 5 cone types of vertices:\\
$\cdot$Type 0: point 1 (which corresponds to 1/2);\\
$\cdot$Type 1: black vertices;\\
$\cdot$Type 2: grey vertices excluding 1;\\
$\cdot$Type 3: white vertices on the tree;\\
$\cdot$Type 4: white vertices not on the tree.\\
Let's write $\Lambda_{i}^{n} = \sum\limits_{g \in C_i, |g| = n}{g}$, where $C_i$ is the i-th
cone type.
Then one gets recurrent relations:
\begin{equation}\label{rec1}
\begin{split}
\Lambda^n_0 &= \Lambda^{n-1}_1a+\Lambda^{n-1}_3a^{-1}\\
\Lambda^n_1 &= \Lambda^{n-1}_1a\\
\Lambda^n_2 &= \Lambda^{n-1}_1a + \Lambda^{n-1}_2b + \Lambda^{n-1}_3a^{-1} \\
\Lambda^n_3 &= \Lambda^{n-1}_2b + \Lambda^{n-1}_4(a^{-1}+b^{-1})\\
\Lambda^n_4 &= \Lambda^{n-1}_4(
a^{-1}+b^{-1})
\end{split}
\end{equation}
Denote $L_i^n$ the number of geodesics of length $n$
starting from a vertex of type $i$, leading to different points, i.e. $L_i^n = \partial B(x_i, n)$ for $x_i$ being a vertex of type $i$. Then recurrent relations are:

\begin{equation}\label{rec2}
\begin{split}
L^n_0 &= L^{n-1}_1 + L^{n-1}_3\\
L^n_1 &= L^{n-1}_1\\
L^n_2 &= L^{n-1}_1 + L^{n-1}_2 + L^{n-1}_3 \\
L^n_3 &= L^{n-1}_2 + L^{n-1}_4\\
L^n_4 &= L^{n-1}_4
\end{split}
\end{equation}
 Let $\Lambda^n=(\Lambda^n_0, \Lambda^n_1, \Lambda^n_2,\Lambda^n_3, \Lambda^n_4)^T$ and
$L^n=(L^n_0, L^n_1, L^n_2,L^n_3, L^n_4)^T$.
We compute $\tilde{L}(z) = \sum_{n=0}^{\infty}\Lambda^nz^n$ - extended complete geodesic growth function and
$\widehat{L}(z) =\sum_{n=0}^{\infty}{L^nz^n}$ - geodesic orbit growth function
(if two geodesics lead to the same point, they are counted as one).
\begin{comment}$ \textbf{A} = \left[
\begin{array}{ccccc}
0 & a & 0 & a^{-1} & 0 \\
0 & a & 0 & 0 & 0 \\
0 & a & b & a^{-1} & 0 \\
0 & 0 & b & 0 & a^{-1}+b^{-1} \\
0 & 0 & 0 & 0 & a^{-1}+b^{-1}
\end{array}
\right]$\\

$ A = \left[
\begin{array}{ccccc}
0 & 1 & 0 & 1 & 0 \\
0 & 1 & 0 & 0 & 0 \\
0 & 1 & 1 & 1 & 0 \\
0 & 0 & 1 & 0 & 2 \\
0 & 0 & 0 & 0 & 2
\end{array}
\right]$,
$ B = \left[
\begin{array}{ccccc}
0 & 1 & 0 & 1 & 0 \\
0 & 1 & 0 & 0 & 0 \\
0 & 1 & 1 & 1 & 0 \\
0 & 0 & 1 & 0 & 1 \\
0 & 0 & 0 & 0 & 1
\end{array}
\right]$.\\
\\
\\
\end{comment}
By recurrent formulas, we have
 $$\tilde{L}(z) = \sum_{n=0}^{\infty}\textbf{A}^n\Lambda_0z^n = (I_5-\textbf{A}z)^{-1}\tilde\Lambda_0 \,\,\text{and}\,\,\widehat{L}(z) = (I_5-Bz)^{-1}L_0,$$
 where $\tilde\Lambda_0 = (e, e, e, e, e)^T,$ $L_0 = (1, 1, 1, 1, 1)^T$ and transition matrices $A,B, \textbf{A}$ are obtained from recurrent relations \ref{rec1}, \ref{rec2}.
 In particular, geodesic growth function is given by the first coordinate of vector-function
 $$\overline{l}(z) = (I_5-Az)^{-1}{\Lambda_0},$$ where ${\Lambda_0} = L_0 = (1, 1, 1, 1, 1)^T.$
 Performing calculations, one gets
 $$\overline{l}(z) = \left(\frac1{1 - 2 z}, \frac1{1 - z}, \frac1{1 - 3 z + 2 z^2}, \frac1{1 - 3 z +
 2 z^2}, \frac1{1 - 2 z}\right).$$ So, $\frac{1}{1-2z}$ is the geodesic growth function for our Schreier graph.
 Also we get $$\widehat{L}(z) = \left(\frac{1 + z}{1 - z - z^2}, \frac1{1 - z}, \frac{1 + z}{1 - 2 z + z^3}, \frac1{
 1 - 2 z + z^3}, \frac1{1 - z}\right).$$
Finally, we see that we've obtained $l(z) = \frac{1}{1-2z}$ and $ \tilde{l}(z) = \frac{1 + z}{1 - z - z^2}$.
$$\tilde{l}(z) = \frac{1 + z}{1 - z - z^2}
= \frac{\varphi^2}{\sqrt{5}(1-\varphi z)} -
\frac{\hat{\varphi}^2}{\sqrt{5}(1-\hat{\varphi} z)},$$
where $ \varphi=\frac{\sqrt{5}+1}{2}$,
$ \hat{\varphi}=\frac{-\hat{\sqrt{5}}+1}{2}$.
So, $L_n = |\partial B(p,n)| = \frac{\varphi^{n+2} - \hat{\varphi}^{n+2}}{\sqrt{5}}$.
$|B(p,n)| = \sum_{k \leq n}{L_k} =
\frac{\varphi^{n+3} - \varphi^2}{\sqrt{5}(\varphi-1)} -
\frac{\hat{\varphi}^{n+3} - \hat{\varphi}^2}{{\sqrt{5}(\hat{\varphi}}-1)}$.
\section*{Acknowledgements}
The author would like to thank his advisor Kate Jushchenko for her guidance and insightful discussions. The author would like to acknowledge Labex Milyon for funding the research.

\end{document}